\documentclass{article}
\usepackage{amsmath, amsfonts, amssymb, amsthm, mathabx}

 \usepackage{color, pdfcolmk}

\newcommand{\RR}{{\mathbb R}}
\newcommand{\NN}{{\mathbb N}}
\newcommand{\EE}{{\mathbb E}}
\newcommand{\SF}{{\mathcal S}}
\newcommand{\RI}{{\mathcal R}}
\newcommand{\LL}{{\mathcal L}}

\numberwithin{equation}{section}

\theoremstyle{plain}
\newtheorem{theorem}{Theorem}[section]
\newtheorem{lemma}[theorem]{Lemma}
\newtheorem{corollary}[theorem]{Corollary}

\theoremstyle{definition}
\newtheorem{definition}[theorem]{Definition}

\def\ds{\displaystyle}

\begin{document}

\title{Integrals with values in Banach spaces and locally convex spaces}

\author{\small{Piotr Mikusi\'nski}\\
\small{University of Central Florida, Orlando, Florida}}
\date{}

\maketitle

\section{Introduction}

The purpose of this article is to present the construction and basic properties of the Bochner integral on abstract measure spaces.  The approach presented here is based on the ideas from \cite{JM} where the Bochner integral is presented for functions defined on $\mathbb{R}^N$. This method seems to be quite flexible and works well for a number of different constructions in analysis: the Lebesgue integral for functions on $\mathbb{R}^N$ (see \cite{JM}, as well as \cite{Hilbert}, \cite{JMPM}, and \cite{PMMT}), the Bochner integral for functions on $\mathbb{R}^N$  (see \cite{JM}), the Daniell integral (see \cite{Daniell}), and measures on abstract spaces in (see \cite{Measures}).  

In the last section present an extension of the construction to functions with values in a locally convex space.  An extension of the Bochner integral to arbitrary locally convex spaces has been recently presented at \cite{nets}.  It uses nets of simple functions to approximate integrable functions. It is not clear if that construction yields an isomorphic space of integrable functions.

\section{Simple functions}

Let $X$ be a nonempty set and let $(\EE,\|\cdot\|)$ be a Banach space.  

\begin{definition}
A collection $\RI$ of subsets of $X$ is called a {\it ring of subsets} of $X$ if
\begin{itemize}
 \item[] $A,B\in\RI$ implies $A \cup B, A \setminus B \in \RI$.
\end{itemize} 
\end{definition}

\begin{definition}
A map $\mu : \RI \to [0,\infty)$ is called {\it $\sigma$-additive} if for any sequence of disjoint sets $A_1, A_2, \dots \in \RI$ such that $\bigcup_{n=1}^\infty A_n \in \RI$ we have 
$$
\mu\left( \bigcup_{n=1}^\infty A_n \right) = \sum_{n=1}^\infty \mu (A_n).
$$
\end{definition}

Note that $\mu (A)$ is finite for all $A\in \RI$. In what follows we will assume that $\RI$ is a ring of subsets of $X$ and $\mu$ is a $\sigma$-additive measure on $\RI$. It will be convenient to use the same symbol to denote a subset of $X$ and the characteristic function of that set, that is, if $A\subset X$ we will write
$$
A(x) = \left\{ \begin{array}{ll}
1 & \mbox{if } x\in A\\
0 & \mbox{otherwise}\\
\end{array} \right. .
$$

\begin{definition}
 A function $f: X \to \EE$ is called a {\it simple function} if 
\begin{equation}\label{simple}
f=v_1 A_1 + \dots + v_n A_n
\end{equation}
for some $A_1, \dots , A_n\in \RI$ and $v_1, \dots v_n \in \EE$. The vector space of all simple $\EE$-valued functions will be denoted by $\SF(\EE)$. 
\end{definition}

If $f\in \SF(\EE)$ is a simple function, then the function $|f|$ defined by $|f|(x)=\|f(x)\|$ is a simple real valued function, that is, $|f|\in\SF(\RR)$. For the simple function 
$$
f=v_1 A_1 + \dots + v_n A_n,
$$
where $A_1, \dots , A_n\in \RI$ and $v_1, \dots v_n \in \EE$, 
we define
$$
\int f= v_1 \mu(A_1) + \dots + v_n \mu(A_n).
$$
A standard argument shows that this integral is well defined. It follows directly from the definition that the map $\int: \SF(\EE) \to \EE$ is a linear.  For $f\in \SF(\EE)$ we define
$$
\|f\|_1 = \int |f|.
$$
It is easy to see that $\| \cdot \|_1$ is a semi-norm on $\SF(\EE)$. 

\begin{lemma}\label{L1}  For any $A_1, \dots , A_n\in \RI$ and $v_1, \dots v_n \in \EE$ we have 
$$
\left\|v_1 \mu(A_1) + \dots + v_n \mu(A_n)\right\| \leq \int \left|v_1 A_1 + \dots + v_n  A_n \right|.
$$ 
\end{lemma}

\begin{proof} First we observe that
$$
v_1 A_1 + \dots + v_n A_n = u_1 B_1 + \dots + u_m B_m,
$$
for some $u_1, \dots , u_m\in \EE$ and some disjoint $B_1, \dots , B_m\in \RI$. Then
$$
v_1 \mu(A_1) + \dots + v_n\mu( A_n) = u_1 \mu(B_1) + \dots + u_m \mu(B_m),
$$
and
\begin{align*}
 \left\|v_1 \mu(A_1) + \dots + v_n \mu(A_n) \right\| &= \left\|u_1 \mu(B_1) + \dots +u_m \mu(B_m) \right\|\\
&\leq \left\|u_1\mu(B_1) \right\| + \dots + \left\|u_m \mu(B_m)\right\|\\
&= \|u_1\| \mu(B_1) + \dots + \|u_m\| \mu(B_m)\\
&= \int \left( \|u_1\|  B_1 + \dots + \|u_m\| B_m \right) \\
&= \int |u_1  B_1 + \dots +u_m B_m | \\
&= \int \left|v_1 A_1 + \dots + v_n  A_n \right|.
\end{align*}
\end{proof}

From the above lemma we obtain the following useful inequality
$$
\left\| \int f \right\| \leq \|f\|_1
$$
for all $f \in \SF(\EE)$.

\section{The Bochner Integral}

In this section we describe a construction of the Bochner integral on $X$ with respect to the extension of $\mu$ to a complete $\sigma$-additive measure on $X$.  In our approach, the extended measure does not play any role.  On the other hand, it is important that $(X, \RI, \mu)$ can be used to define the space of Lebesgue integrable real valued functions on $X$ with respect to the extension of $\mu$ to a complete $\sigma$-additive measure on $X$ (see \cite{Measures}). We will use $L^1(X, \mu)$ to denote that space.  In proofs involving arguments in $L^1(X, \mu)$  (for example, the proof of  Theorem \ref{Th|f|})  we follow the approach presented in \cite{Daniell}.  

\begin{definition}
A function $f: X \to \EE$ is called \emph{Bochner integrable} if
there exist functions $f_n\in \SF(\EE)$ such that
\begin{enumerate}
	\item[$\mathbb{A}$] $\ds \sum_{n=1}^\infty \|f_n\|_1 < \infty$,
	\item[$\mathbb{B}$] $\ds f(x) = \sum_{n=1}^\infty f_n(x)$ at those points $x\in X$ where $\ds \sum_{n=1}^\infty \|f_n(x)\| < \infty$.
\end{enumerate}

If conditions $\mathbb{A}$ and $\mathbb{B}$ are satisfied we will write 
\begin{equation*}
	f \simeq f_1+ f_2+\dots \quad \text{or} \quad f\simeq \sum_{n=1}^\infty f_n.
\end{equation*} 

The space of all Bochner integrable functions will be denoted by $L^1(X,\mu, \EE)$.
\end{definition}

We are going to define the integral of a Bochner integrable function $f \simeq f_1+ f_2+\dots$ as $\int f = \int f_1+ \int f_2+\dots$.  Note that this definition requires a proof of independence of the integral $\int f$ of a particular expansion of $f$ in a series of simple functions. This proof is not simple and uses some properties of the integral in  $L^1(X, \mu)$. 

\begin{theorem}\label{Th|f|} Let $f \simeq f_1+ f_2+\dots$.  Then
\begin{itemize}
 \item[{\rm (a)}] $|f|\in L^1(X,\mu)$,
 \item[{\rm (b)}] $\ds \int |f| = \lim_{n\to \infty} \int |f_1+ \dots + f_n|$,
 \item[{\rm (c)}] $\ds \left\| \int f_1+ \int f_2+\dots \right\| \leq \int |f|$.
\end{itemize}
\end{theorem}

\begin{proof}
 Let
 $$
 g_n= f_1+ \dots + f_n \; \text{ for } n=1,2, \dots
 $$
 and 
 $$
\varphi_1=|f_1| \; \text{ and } \; \varphi_n = |g_n|-|g_{n-1}| \; \text{ for } n=2,3, \dots.
 $$
 We will show that
 \begin{equation}\label{|f|exp}
 |f| \simeq \varphi_1 +|f_1|-|f_1|+\varphi_2 +|f_2|-|f_2|+ \dots .
\end{equation}
Since
$$
|\varphi_n| = \left| |g_n|-|g_{n-1}| \right| \leq |g_n-g_{n-1}| = |f_n|, 
$$
we have
$$
\|\varphi_1\|_1 +\|f_1\|_1+\|f_1\|_1+\|\varphi_2\|_1 +\|f_2\|_1+\|f_2\|_2+ \dots \leq 3  \sum_{n=1}^\infty \|f_n\|_1 < \infty.
$$
If
$$
|\varphi_1(x)| +\|f_1(x)\|+\|f_1(x)\|+|\varphi_2(x)| +\|f_2(x)\|+\|f_2(x)\|+ \dots
< \infty 
$$
for some $x\in X$, then $\sum_{n=1}^\infty \|f_n(x)\|_1 < \infty$ and consequently $\sum_{n=1}^\infty f_n(x)=f(x)$. Hence
$$
\sum_{n=1}^m \varphi_n(x)=\|g_m(x)\|=\left\|\sum_{n=1}^m f_n(x)\right\| \to \|f(x)\|
$$
as $m\to \infty$. 

From \eqref{|f|exp} we obtain $|f|\in L^1(X,\mu)$. Moreover,
\begin{align*}
\int |f| &= \int \varphi_1 +\int |f_1|-\int |f_1|+\int \varphi_2 +\int |f_2|-\int|f_2|+ \dots\\
&= \lim_{n\to \infty} \int (\varphi_1+ \dots + \varphi_n)\\
&= \lim_{n\to \infty} \int |g_n|\\
&= \lim_{n\to \infty} \int |f_1+ \dots + f_n|.
\end{align*}

Finally, since
$$
\left\|\sum_{n=1}^m \int f_n\right\| \leq \int \left|\sum_{n=1}^m f_n\right|=\int |g_m|=\int | f_1+ \dots + f_m|,
$$
we have
$$
\left\| \int f_1+ \int f_2+\dots \right\| \leq \lim_{n\to \infty} \int |f_1+ \dots + f_n| = \int |f|.
$$
\end{proof}

\begin{corollary}
 If 
 $$
 f \simeq f_1+ f_2+\dots \quad \text{and} \quad f \simeq g_1+ g_2+\dots ,
 $$
 then
 $$
\int f_1+ \int f_2+\dots = \int g_1+ \int g_2+\dots .
 $$
\end{corollary}

\begin{proof}
 If $f \simeq f_1+ f_2+\dots $ and $ f \simeq g_1+ g_2+\dots$, then
 $$
 0 \simeq f_1-g_1+f_2-g_2+ \dots
 $$
 and hence
 $$
 \left\|  \int f_1- \int g_1+  \int f_2- \int g_2 +\dots \right\| \leq 0.
 $$
\end{proof}

Now we can define the integral of a Bochner integrable function.

\begin{definition}
 By the integral of a Bochner integrable function
$$
f \simeq  f_1 + f_2 + \cdots
$$
we mean the element of $\EE$ defined by
$$
\int{f} = \int{f_1} +  \int{f_2} + \cdots .
$$
\end{definition}

\begin{theorem}
 The integral is a linear operator from $L^1(X,\mu, \EE)$ to $\EE$.  Moreover, $\left\| \int f \right\| \leq \int |f|$ for all $f \in L^1(X,\mu, \EE)$ . 
\end{theorem}

\begin{proof} Linearity follows easily from the fact that, if  $f \simeq  f_1 + f_2 + \cdots$, $g \simeq  g_1 + g_2 + \cdots$, and $\lambda\in\RR$, then
$$
f+g \simeq  f_1 +g_1 + f_2 + g_2\cdots \quad \text{and} \quad \lambda f \simeq  \lambda f_1 + \lambda f_2 + \cdots .
$$
The inequality follows from part (c) of Theorem \ref {Th|f|}.
\end{proof}

We complete this section with a proof of Gelfand-Pettis integrability (\cite{Gelfand} and \cite{Pettis}) of Bochner integrable functions. If $\Lambda$ is a bounded linear functional on $\EE$ and $f:X\to \EE$, then the composition of $\Lambda$ and $f$ will be denoted by $\Lambda f$, that is $\Lambda f (x)=\Lambda(f(x))$.

\begin{theorem}
 Let $\Lambda$ be a bounded linear functional on $\EE$. If $f\in L^1(X,\mu, \EE)$, then $\Lambda  f \in L^1(X,\mu, \RR)$ and
 $$
 \Lambda \int f = \int \Lambda  f.
 $$
\end{theorem}

\begin{proof}
 Let $\Lambda$ be a bounded linear functional on $\EE$. 
 
 If $f=v_1A_1+\dots+v_nA_n$ for some $v_1, \dots , v_n \in \EE$ and disjoint $A_1, \dots , A_n \in \RI$, then
\begin{align*}
 \Lambda \int f &= \Lambda \left(\int (v_1A_1+\dots+v_nA_n)\right)\\
 &= \Lambda \left(v_1\mu(A_1)+\dots+v_n\mu(A_n)\right)\\
 &= \Lambda v_1 \mu(A_1)+\dots+\Lambda v_n \mu(A_n)\\
 &= \int \left(\Lambda v_1 A_1+\dots+\Lambda v_n A_n\right)=\int \Lambda f
\end{align*}
and
\begin{align*}
 \| \Lambda  f\|_1 &= \int |\Lambda (v_1A_1+\dots+v_nA_n)|\\
&= |\Lambda v_1| \mu(A_1)+\dots+|\Lambda v_n| \mu(A_n)\\
&\leq \|\Lambda\| \left(  \|v_1\| \mu(A_1)+\dots+ \|v_n\| \mu(A_n)\right)=  \|\Lambda\| \|f\|_1.
\end{align*}

If $f \simeq  f_1 + f_2 + \cdots$, then 
\begin{equation}\label{Lexpension}
\Lambda  f \simeq \Lambda  f_1 + |f_1|-|f_1|+\Lambda  f_2 +|f_2|-|f_2|+ \cdots .
\end{equation}
Indeed, we have
$$
\|\Lambda  f_1\|_1 + 2\|f_1\|_1+\|\Lambda  f_2\|_1 +2\|f_2\|_1+ \cdots \leq (\|\Lambda\|+2)\sum_{n=1}^\infty \|f_n\|_1<\infty
$$
and, if 
$$
|\Lambda(f_1(x))| + 2\|f_1(x)\|+|\Lambda (f_2(x))| +2\|f_2(x)\|+ \cdots < \infty
$$
for some $x\in X$, then $\sum_{n=1}^\infty \|f_n(x)\|<\infty$ and thus $f(x)=\sum_{n=1}^\infty f_n(x)$. Consequently,
$$
\Lambda f(x)=\sum_{n=1}^\infty \Lambda f_n(x)
$$
for that $x\in X$.

From \eqref{Lexpension} we get both $\Lambda  f \in L^1(X,\mu, \RR)$ and $\Lambda \int f = \int \Lambda  f$.
\end{proof}

\section{$L^1(X,\mu, \EE)$ as a Banach space}

It is important the tools of normed spaces can be applied to Bochner integrable functions. However, we need to deal with the usual problem, that is, $\|f\|_1=\int |f|$ is not a norm in $L^1(X,\mu, \EE)$, since $\int |f|=0$ does not imply $f=0$. The problem can be solved by identifying functions that are equal almost everywhere.

If $f,g \in L^1(X,\mu, \EE)$ and $\int |f-g|=0$, then we write $f \sim g$.  It is clear that $\sim$ is an equivalence relation. Let $\LL^1(X,\mu, \EE)$ be the space of equivalence classes, that is, $\LL^1(X,\mu, \EE)=L^1(X,\mu, \EE)/\sim$. It is easy to check that $\|f\|_1=\int |f|$ is a norm in $\LL^1(X,\mu, \EE)$. The difference between $\LL^1(X,\mu, \EE)$ and $L^1(X,\mu, \EE)$ is usually ignored.  It is important to be aware of this difficulty and to carefully interpret statements about $L^1(X,\mu, \EE)$. Then it should not lead to any problems. 

If $\|f_n -f\|_1\to 0$, then we will say that the sequence $(f_n)$ converges to $f$ {\it in norm} and write $f_n \to f$ i.n..

It is our goal to show that $L^1(X,\mu, \EE)$ is complete with respect to $\|\cdot\|_1$. The first step is the following technical lemma.

\begin{lemma}  \label{lem2.5.1} If $f \in L^1(X, \mu, \EE)$, then for
every $\varepsilon > 0$ there exists a sequence of simple functions 
$(f_n)$ such that $f \simeq  f_1  +  f_2  + \dots $ and 
$\sum^\infty_{n=1}\int|f_n|  \leq  \int|f|  +  \varepsilon$.
\end{lemma}

\begin{proof}  Let $f \simeq  g_1  +  g_2  + \dots $ be an arbitrary expansion of $f$ in a series of simple functions. Then there exists an $n_0 \in \mathbb N$ such that 
$\sum^\infty_{n=n_0+1}\int|g_n|  <  \frac{\varepsilon}{2}$.
Define 
$$
f_1  =  g_1  + \dots   +  g_{n_0} \;\; \text{and} \;\; f_n  =  g_{n_0+n-1} \;\;\text{ for }\;\; n  \geq  2. 
$$
Then obviously $f \simeq  f_1  +  f_2  + \dots $. Since $\int |f_1| - \int |f| \leq \int |f_1 - f|$ and $f-f_1 \simeq f_2+f_3+\dots $, we get
$$
\int |f_1| - \int |f| \leq \sum_{n=2}^\infty \int |f_n|
$$
and hence
$$
\int |f_1| - \sum_{n=2}^\infty \int |f_n| \leq \int |f|. 
$$
Consequently,
\begin{align*}
\sum^\infty_{n=1}\int|f_n| &= \int|f_1| + \sum_{n=2}^\infty \int |f_n|\\
&= \int|f_1| - \sum_{n=2}^\infty \int |f_n| +2 \sum_{n=2}^\infty \int |f_n|\\
&\leq \int |f| + 2 \sum_{n=2}^\infty \int |f_n|\\
&= \int |f| + 2 \sum^\infty_{n=n_0+1}\int |g_n|\\ 
&<   \int|f|  +  \varepsilon. 
\end{align*}
\end{proof}

Now we generalize the symbol $\simeq$ to series of arbitrary Bochner integrable functions.

Let $f_1, f_2, \dots \in L^1(X,\mu, \EE)$ and let $f:X\to\EE$ be an arbitrary function. If
\begin{enumerate}
	\item[$\mathbb{A}$] $\ds \sum_{n=1}^\infty \|f_n\|_1 < \infty$ and
	\item[$\mathbb{B}$] $\ds f(x) = \sum_{n=1}^\infty f_n(x)$ at those points $x\in X$ where $\ds \sum_{n=1}^\infty \|f_n(x)\| < \infty$,
\end{enumerate}
then we write 
\begin{equation*}
	f \simeq f_1+ f_2+\dots \quad \text{or} \quad f\simeq \sum_{n=1}^\infty f_n.
\end{equation*} 

\begin{theorem}\label{ThExp}

If $f_1,f_2,\ldots\in L^1(X,\mu, \EE)$ and $f \simeq f_1 + f_2 + \cdots,$ then $f\in L^1(X,\mu, \EE)$,
$$
f_1+f_2+\ldots = f \text{ i.n.}
$$
and
	$$\int f = \int f_1 + \int f_2 + \cdots.$$
	
\end{theorem}

\begin{proof}

Let $\varepsilon>0$ be arbitrary and let $\varepsilon_1 + \varepsilon_2 + \cdots$ be a series of positive numbers whose sum is $\varepsilon$. By Lemma \ref{lem2.5.1}, we can choose expansions
\begin{equation}\label{b20}
f_i \simeq f_{i1} + f_{i2} + \cdots \quad (i=1,2,\ldots),
\end{equation}
where $f_{ij}$ are simple functions such that
\begin{equation}\label{b21}
\int |f_{i1}| + \int |f_{i2}| + \cdots <  \int |f_i|+\varepsilon_i
\end{equation}
for all $i\in\NN$.  Let
\begin{equation}\label{b22}
g_1 + g_2 + \cdots
\end{equation}
be a series of simple functions which is composed of all the series in \eqref{b20}. Then from \eqref{b21} we obtain
\begin{equation}\label{b23}
\int |g_1| + \int |g_2| + \cdots < M + \varepsilon_1 + \varepsilon_2 + \cdots ,
\end{equation}
where $M = \int |f_1| + \int |f_2| + \cdots.$ Moreover, if the series \eqref{b22} converges absolutely at a point $x\in X$, then each of the series in \eqref{b20} converges absolutely at that point, and consequently 
$$
g_1(x) + g_2(x) + \cdots = f_1(x) + f_2(x) + \cdots = f(x)
$$
at that $x$. 
This proves that $f$ is Bochner integrable and 
$$
\int f = \int g_1 + \int g_2 + \cdots = \int f_1 + \int f_2 + \cdots.
$$
Moreover, since for every $n\in\NN$,
$$
f-f_1-\dots - f_n \simeq f_{n+1} + f_{n+2} + \dots,
$$
we have
$$
\|f-f_1-\dots - f_n\|_1 \leq \sum_{k=n+1}^\infty \|f_k\|_1 \to  0
$$
as $n\to \infty$, which means that $f_1+f_2+\ldots = f$ i.n..
\end{proof}

\begin{corollary}\label{somofser} If $f_1, f_2, \dots \in L^1(X,\mu, \EE)$ and $\ds \sum_{n=1}^\infty \|f_n\|_1 < \infty$, then there exists $f\in L^1(X,\mu, \EE)$ such that 
$f \simeq f_1+ f_2+\dots$. 
\end{corollary}

\begin{proof}
 Let $f_1, f_2, \dots \in L^1(X,\mu, \EE)$ be such that $\ds \sum_{n=1}^\infty \|f_n\|_1 < \infty$. Define
$$
f(x)=\begin{cases} 
\sum^\infty_{n=1}f_n(x) & \text{ whenever }
\sum^\infty_{n=1}\|f_n(x)\|  <  \infty,\\
0 & \text{ otherwise.}
\end{cases}  
$$
Then $f \simeq f_1+ f_2+\dots$ and,
by Theorem \ref{ThExp}, $f\in L^1(X,\mu, \EE)$.
 \end{proof}

\begin{theorem}\label{thm2.8.1} The space $(L^1(X, \mu, \EE),\|\cdot\|_1)$ is complete.
\end{theorem}

\begin{proof}   We will prove that every absolutely convergent series converges in norm. If $\sum^\infty_{n=1}\|f_n\|_1 < \infty$ for some $f_n \in L^1(X, \mu, \EE)$, then, by Corollary \ref{somofser}, there exists an $f \in L^1(X, \mu, \EE)$ such that $f \simeq \sum^\infty_{n=1}f_n$. This in turn implies, by Theorem \ref{ThExp}, that the series $\sum^\infty_{n=1}f_n$ converges to $f$ in norm.
\end{proof}

\section{Convergence almost everywhere} 

A set $A \subseteq \mathbb X$ is called a {\it null set} or a {\it set of measure zero} if its characteristic function is integrable, that is, $\chi_A\in L^1(X,\mu)$ and $\int \chi_A =0$. Every subset of a null set is a null set.  A countable union of null sets is a null set. 

\begin{definition}\label{def2.7.2}  If
for two functions $f,g: X \to \EE$ the set of all
$x \in X$ for which $f(x) \neq g(x)$ is a null set,
then we say that $f$ {\it equals $g$ almost everywhere} and write $f = g$ a.e.. 
\end{definition} 

\begin{theorem} \label{thm2.7.2} $f = g$ a.e. if and only if $\int|f-g| = 0$.   
\end{theorem}

\begin{proof}  Let $Z=\{x \in X : f(x) \neq g(x)\}$ and let $h$ be the characteristic function of $Z$.
  
If $f = g$ a.e., then $\int|h| = \int h = 0$. 
Therefore 
$$
|f - g| \simeq h + h + \dots  ,
$$
which implies $\int|f-g| = 0$. 

Conversely, if $\int|f-g| = 0$, then
$$
h \simeq |f-g| + |f-g| + \dots  ,
$$
and hence $\int h = 0$.  This shows that $Z$ is a null set, that is,
$f = g$ a.e..
 
\end{proof}

\begin{theorem} \label{thm2.7.3}  Suppose $f_n \rightarrow f$ i.n. Then $f_n \rightarrow g$  i.n. if and only if $f = g$ a.e..
\end{theorem}

\begin{proof}  If $f_n \rightarrow f$ i.n. and $f = g$ a.e., then
$$
\|f_n-g\|_1 = \int|f_n-g| \leq \int|f_n-f| + 
\int|f-g| = \int|f_n-f| = \|f_n-f\|_1 
\rightarrow 0.
$$

If $f_n \rightarrow f$ i.n. and $f_n \rightarrow g$ i.n., then 
$f_n-f_n \rightarrow f-g$ i.n..  This implies  
$$
\int|f-g|  =  \int|f_n-f_n-f + g|  \rightarrow  0,
$$
completing the proof.
\end{proof}

\begin{definition} We say that a sequence of functions $f_1,  f_2,\dots : X \to \EE$  {\it converges to $f$ almost everywhere} and write $f_n \rightarrow f$ a.e., if $f_n(x) \rightarrow f(x)$ for every $x$ except a null set.
\end{definition} 

The following properties of convergence almost everywhere are immediate consequences of the definition:

{\it
If $ f_n \rightarrow f$ a.e. and
$\lambda \in \mathbb R$, then 
$\lambda f_n \rightarrow \lambda f$ a.e.

If $f_n \rightarrow f$ a.e. and $g_n \rightarrow g$ a.e.,
then  $f_n + g_n \rightarrow f + g$ a.e.

If $f_n \rightarrow f$ a.e., then $|f_n| \rightarrow |f|$
a.e.
}

\begin{theorem} \label{thm2.7.4} Suppose $f_n \rightarrow f$ a.e. Then $f_n \rightarrow g$ a.e. if and only if $f = g$ a.e.
\end{theorem}

\begin{proof}  If $f_n \rightarrow f$ a.e. and $f_n \rightarrow g$ a.e., then $f_n-f_n \rightarrow f-g$ a.e., which means
that $f-g = 0$ a.e..

Now let $A=\{ x \in X : f_n(x)\nrightarrow f(x)\}$ and $B=\{ x \in X : g_n(x)\nrightarrow g(x)\}$.  Then $A$ and $B$ are null sets and so is $A\cup B$. Since $f_n(x)\rightarrow g(x)$ for every $x$ not in $A\cup B$, we have  $f_n \rightarrow g$ a.e..
\end{proof}

\begin{theorem} \label{thm2.7.5} If $f_1, f_2, \dots \in L^1(X,\mu,\EE)$ and $\ds \sum_{n=1}^\infty \|f_n\|_1 < \infty$, then the series 
$f_1 + f_2 + \dots $ converges almost everywhere.
\end{theorem}

\begin{proof}  By Corollary \ref{somofser}, there exists a function 
$f \in L^1(X,\mu,\EE)$ such that $f \simeq f_1 + f_2 + \dots $. Since $f(x) = \sum^\infty_{n=1}f_n(x)$ for every $x\in X $ such that $\sum^\infty_{n=1}\|f_n(x)\| < \infty$, it suffices to show that the set of all points $x \in X$ for which the series
$\sum^\infty_{n=1}\|f_n(x)\|$ is not absolutely convergent is a null set.  Let $g$ be the characteristic function of that set. Then $g \simeq f_1 - f_1 + f_2 - f_2 + \dots  $, and consequently
$$
\int|g| = \int g =\int f_1 - \int f_1 + \int f_2 - \int f_2 + \dots  = 0. 
$$
\end{proof}

The above theorem leads us to an important corollary.

\begin{corollary} \label{cor2.7.1} If $f \simeq f_1 + f_2 + \dots $ , then 
$f = f_1 + f_2 + \dots $ a.e. 
\end{corollary}

In general, convergence in norm does not imply convergence almost everywhere and convergence almost everywhere does not imply convergence in norm. It turns out that for absolutely convergent series in $L^1(X,\mu,\EE)$ both types of convergence are equivalent.

\begin{theorem} \label{thm2.7.6}  Let $f_1, f_2, \dots  \in L^1(X,\mu,\EE)$  and $\ds \sum_{n=1}^\infty \|f_n\|_1 < \infty$.  Then \\ 
$f = f_1 + f_2 + \dots $ a.e. if and only if 
$f = f_1 + f_2 + \dots $ i.n.
\end{theorem}

\begin{proof}  By Corollary \ref{somofser}, there exists a function 
$g \in L^1(\mathbb R^N)$ such that $g \simeq f_1 + f_2 + \dots $ Then, by Theorem \ref{ThExp} we have $g = f_1 + f_2 + \dots $ i.n. and, by Corollary \ref{cor2.7.1},  we have $g = f_1 + f_2 + \dots $  a.e..

Now, if $f = f_1 + f_2 + \dots $ a.e., then $f = g$ a.e., by Theorem \ref{thm2.7.4}. Hence $f = f_1 + f_2 + \dots $  i.n., by
Theorem \ref{thm2.7.3}.

Conversely, if $f = f_1 + f_2 + \dots $  i.n., then  $f = g$ a.e., by Theorem \ref{thm2.7.3}. Hence $f = f_1 + f_2 + \dots $  a.e., by
Theorem \ref{thm2.7.4}.
\end{proof}

\section{The Dominated Convergence Theorem}

First we prove the following useful theorem that sheds more light on the relationship between convergence almost everywhere and convergence in norm.

\begin{theorem}\label{thm2.8.2} If $f_n \rightarrow f$ i.n., then there exists a subsequence $(f_{p_n})$  of $(f_n)$ such that $f_{p_n} \rightarrow f$ a.e..
\end{theorem}

\begin{proof}  Since $\|f_n-f\|_1 \rightarrow 0$, there exists an increasing sequence of positive integers $(p_n)$ such that 
$\|f_{p_n}-f\|_1 < 2^{-n}$.  Then
$$
\|f_{p_{n+1}}-f_{p_n}\|_1 \leq \|f_{p_{n+1}}-f\|_1 + \|f-f_{p_n}\|_1 <
\frac{3}{2^{n+1}}
$$
and consequently
$$
\|f_{p_1}\|_1 + \|f_{p_2}-f_{p_1}\|_1 + \|f_{p_3}-f_{p_2}\|_1 + \dots   < \infty.
$$
Thus, there exists a $g \in L^1(X,\mu,\EE)$ such that 
$$
g \simeq f_{p_1} + (f_{p_2}-f_{p_1}) + (f_{p_3}-f_{p_2}) + \dots ,
$$
and, by Corollary \ref{cor2.7.1},
$$
g = f_{p_1} + (f_{p_2}-f_{p_1}) + (f_{ p_3}-f_{p_2}) + \dots  \text{ a.e.} \; .
$$
This means $f_{p_n} \rightarrow g$ a.e.  Since also
$f_{p_n} \rightarrow g$ i.n.  and $f_{p_n} \rightarrow f$ i.n., we conclude $f = g$ a.e., by Theorem \ref{thm2.7.3}.  Therefore  $f_{p_n} \rightarrow f$ a.e., by Theorem \ref{thm2.7.4}.  
\end{proof}

Note that the above result can be easily obtained from the same result for real valued functions.  Indeed, if $f_n \rightarrow f$ i.n., then $|f_n-f|\to 0$ i.n. and thus $|f_{p_n}-f|\to 0$ a.e. for some increasing sequence of indices $(p_n)$.  But then $f_{p_n}-f\to 0$ a.e..

\begin{lemma}\label{DomiLemma}
Let $f_1, f_2, \dots  \in L^1(X,\mu,\EE)$. If $f_n \to 0$ a.e. and there exists a function $h\in L^1(X,\mu,\RR)$ such that $|f_n| \leq h$ for every $n \in \mathbb N$, then $f_n \to 0$ i.n..
\end{lemma}

\begin{proof}
For $m,n = 1, 2,\dots $, define
$$
g_{m,n} = \max \{|f_m|,\dots ,|f_{m+n}|\}.
$$
Note that $g_{m,n}\in L^1(X,\mu,\RR)$ for all $m,n\in\NN$. For every $m \in \mathbb N$, the sequence 
$(g_{m,1}, g_{m,2},\dots )$  is non-decreasing and, since
$$
\left |\int g_{m,n} \right| = \int g_{m,n} \leq 
\int h < \infty,
$$
there is a function $g_m\in L^1(X,\mu,\RR)$ such that $g_{m,n} \rightarrow g_m$   a.e. as $n \rightarrow \infty$. 

The sequence $(g_n)$ is non-increasing and $g_n \rightarrow 0$ a.e., since $f_n \rightarrow 0$ a.e.. By the Monotone Convergence Theorem, $g_n \rightarrow 0$ i.n. and thus
$$
\int|f_n| \leq \int g_n \rightarrow 0. 
$$
\end{proof}

\begin{theorem}[The Dominated Convergence Theorem]\label{thm2.8.4}  If a sequence of functions $f_n\in L^1(X,\mu,\EE)$ converges almost everywhere to a function $f$ and there exists a function $h\in L^1(X,\mu,\RR)$ such that $|f_n| \leq h$ for every $n \in \mathbb N$, then 
$f\in L^1(X,\mu,\EE)$ and $f_n \rightarrow f$ i.n..
\end{theorem}

\begin{proof}  
If $(p_n)$ is an increasing sequence of positive integers, then
$$
h_n = f_{p_{n+1}}-f_{p_n} \rightarrow 0 \text{ a.e.}
$$
and $|h_n| \leq 2h$ for every $n \in \mathbb N$.  By Lemma \ref{DomiLemma}, $h_n \rightarrow 0$ i.n.. This shows that the sequence $(f_n)$ is a Cauchy sequence in
$L^1(X,\mu,\EE)$ and therefore it converges in norm to some
$g \in L^1(X,\mu,\EE)$, by Theorem \ref{thm2.8.1}.  On the other
hand, by Theorem \ref{thm2.8.2}, there exists an increasing sequence of positive integers $(q_n)$ such that $f_{q_n} \rightarrow g$ a.e. But
$f_{q_n} \rightarrow f$ a.e., and thus $g = f$ a.e..  This, in
view of Theorem \ref{thm2.7.3}, implies $f_n \rightarrow f$ i.n.. 
\end{proof}

\section{Integrals with values in locally convex spaces}

In tho section we extend the presented construction to functions with values in locally spaces.  As before, $X$ is a nonempty set with a ring of subsets $\RI$ and $\mu$ is a $\sigma$-additive measure $\RI$. Now $\EE$ is a complete locally convex space with the topology defined be a family of seminorms $\|\cdot\|_\alpha$ with $\alpha \in I$, where $I$ is an index set. If $f\in \SF(\EE)$ is a simple function and $\alpha\in I$, then the function $\|f\|_\alpha$ defined by $\|f\|_\alpha(x)=\|f(x)\|_\alpha$ is a simple real valued function, that is, $\|f\|_\alpha\in\SF(\RR)$.

For $\alpha \in I$, let $(\EE_\alpha, \|\cdot\|_\alpha)$ be the quotient normed space.  If $f: X \to \EE$, then by $\pi_\alpha f : X \to \EE_\alpha$ we will denote the composition of $f$ with the quotient map from $\EE$ to $\EE_\alpha$. Note that $\|f\|_\alpha=\|\pi_\alpha f\|_\alpha$.

 For $f\in \SF(\EE)$ and $\alpha\in I$ we define
$$
\vvvert f \vvvert_\alpha = \int \|f\|_\alpha.
$$
It is easy to see that $\vvvert \cdot \vvvert_\alpha$ is a semi-norm on $\SF(\EE)$. Note that $\int \|f\|_\alpha = \int \|\pi_\alpha f\|_\alpha$. From Lemma \ref{L1} we obtain the following useful inequality
$$
\left\| \int f \right\|_\alpha \leq \vvvert f \vvvert_\alpha
$$
for all $f \in \SF(\EE)$ and $\alpha\in I$.

\begin{definition}
A function $f: X \to \EE$ is called integrable if
there exist functions $f_n\in \SF(\EE)$ such that,  for every $\alpha\in I$,
\begin{enumerate}
	\item[$\mathbb{A}$] $\ds \sum_{n=1}^\infty \int \| f_n \|_\alpha < \infty$,
	\item[$\mathbb{B}$] If $\ds  \sum_{n=1}^\infty \|f_n(x)\|_\alpha< \infty$, then $\ds \lim_{m\to\infty} \left\| f(x)- \sum_{n=1}^mf_n(x)\right\|_\alpha = 0$.
\end{enumerate}
If conditions $\mathbb{A}$ and $\mathbb{B}$ are satisfied we will write 
\begin{equation*}
	f \simeq f_1+ f_2+\dots \quad \text{or} \quad f\simeq \sum_{n=1}^\infty f_n.
\end{equation*} 
The space of all integrable functions will be denoted by $L^1(X,\mu, \EE)$.
\end{definition}

\begin{lemma}\label{Bochner} If $f \simeq f_1+ f_2+\dots$, then $\pi_\alpha f \simeq \pi_\alpha f_1+ \pi_\alpha f_2+\dots$ for every $\alpha \in I$. Consequently, if $f: X \to \EE$ is integrable, then $\pi_\alpha f$ is Bochner inetgrable for every $\alpha \in I$.
\end{lemma}

\begin{proof} Let $\alpha \in I$.
If $f: X \to \EE$ is integrable, then
there exist functions $f_n\in \SF(\EE)$ such that $ \sum_{n=1}^\infty \int \| f_n \|_\alpha < \infty$ and such that $ \lim_{m\to\infty} \left\| f(x)- \sum_{n=1}^mf_n(x)\right\|_\alpha = 0$ at those points $x\in X$ where $ \sum_{n=1}^\infty \|f_n(x)\|_\alpha< \infty$. Then $\pi_\alpha f_n\in \SF(\EE_\alpha)$, $ \sum_{n=1}^\infty \int \| \pi_\alpha f_n \|_\alpha < \infty$, and $ \pi_\alpha f(x)= \sum_{n=1}^\infty \pi_\alpha f_n(x)$ at those points $x\in X$ where $ \sum_{n=1}^\infty \|\pi_\alpha f_n(x)\|_\alpha< \infty$.  Consequently, $\pi_\alpha f \simeq \pi_\alpha f_1+ \pi_\alpha f_2+\dots$ and $\pi_\alpha f$ is inetgrable.
\end{proof}

\noindent{\bf Question 1:} If $f: X \to \EE$ and $\pi_\alpha f$ is Bochner inetgrable for every $\alpha \in I$, is $f$ integrable?

\begin{theorem}\label{Th||f||} Let $f \simeq f_1+ f_2+\dots$.  Then, for every $\alpha \in I$, we have 
\begin{itemize}
 \item[{\rm (a)}] $\|f\|_\alpha\in L^1(X,\mu)$,
 \item[{\rm (b)}] $\ds \int \|f\|_\alpha = \lim_{n\to \infty} \int \|f_1+ \dots + f_n\|_\alpha$,
 \item[{\rm (c)}] $\ds \left\| \int f_1+ \int f_2+\dots \right\|_\alpha \leq \int \|f\|_\alpha$.
\end{itemize}
\end{theorem}

\begin{proof}
 Let $f \simeq f_1+ f_2+\dots$ and $\alpha \in I$.  Since $\|f\|_\alpha=\|\pi_\alpha f\|_\alpha$, we obtain $\|f\|_\alpha\in L^1(X,\mu)$ by Lemma \ref{Bochner}.  Moreover, since
$$
\pi_\alpha f \simeq \pi_\alpha f_1+ \pi_\alpha f_2+\dots,
$$
we have
$$
\int \| f\|_\alpha = \int \|\pi_\alpha f\|_\alpha = \lim_{n\to \infty} \int \|\pi_\alpha f_1+ \dots + \pi_\alpha f_n\|_\alpha = \lim_{n\to \infty} \int \|f_1+ \dots + f_n\|_\alpha.
$$
Finally, since
$$
\left\|\sum_{n=1}^m \int f_n\right\|_\alpha = \left\| \int \sum_{n=1}^m \pi_\alpha f_n\right\|_\alpha \leq \int \left\|\sum_{n=1}^m \pi_\alpha f_n\right\|_\alpha = \int \left\|\sum_{n=1}^m f_n\right\|_\alpha,
$$
we have
$$
\left\| \int f_1+ \int f_2+\dots \right\|_\alpha \leq \lim_{n\to \infty} \int \|f_1+ \dots + f_n\|_\alpha = \int \|f\|_\alpha.
$$
\end{proof}

\begin{corollary}
 If 
 $$
 f \simeq f_1+ f_2+\dots \quad \text{and} \quad f \simeq g_1+ g_2+\dots ,
 $$
 then
 $$
\int f_1+ \int f_2+\dots = \int g_1+ \int g_2+\dots .
 $$
\end{corollary}

\begin{proof}
 If $f \simeq f_1+ f_2+\dots $ and $ f \simeq g_1+ g_2+\dots$, then
 $$
 0 \simeq f_1-g_1+f_2-g_2+ \dots
 $$
 and hence
 $$
 \left\|  \int f_1- \int g_1+  \int f_2- \int g_2 +\dots \right\|_\alpha \leq 0
 $$
 for every $\alpha \in I$.
\end{proof}

\begin{definition}
 By the integral of an integrable function
$$
f \simeq  f_1 + f_2 + \cdots
$$
we mean the element of $\EE$ defined by
$$
\int{f} = \int{f_1} +  \int{f_2} + \cdots .
$$
\end{definition}

For $f\in L^1(X,\mu, \EE)$ and $\alpha \in I$ we can define $\vvvert f \vvvert_\alpha = \int \|f\|_\alpha$. These seminorms induce a locally convex topology on $L^1(X,\mu, \EE)$.

\begin{theorem}
 The integral is a linear operator from $L^1(X,\mu, \EE)$ to $\EE$.  Moreover, $\left\| \int f \right\|_\alpha \leq \int \|f\|_\alpha$ for all $f \in L^1(X,\mu, \EE)$ and $\alpha \in I$. 
\end{theorem}

\begin{proof} Linearity follows easily from the fact that, if  $f \simeq  f_1 + f_2 + \cdots$, $g \simeq  g_1 + g_2 + \cdots$, and $\lambda\in\RR$, then
$$
f+g \simeq  f_1 +g_1 + f_2 + g_2\cdots \quad \text{and} \quad \lambda f \simeq  \lambda f_1 + \lambda f_2 + \cdots .
$$
The inequality follows from part (c) of Theorem \ref {Th|f|}.
\end{proof}

\noindent{\bf Question 2:} Is $L^1(X,\mu, \EE)$ complete?

\vspace{3mm}

Finally we show  Gelfand-Pettis integrability (\cite{Gelfand} and \cite{Pettis}) of elements of $L^1(X,\mu, \EE)$.

\begin{theorem}
 Let $\Lambda$ be a bounded linear functional on $\EE$. If $f\in L^1(X,\mu, \EE)$, then $\Lambda  f \in L^1(X,\mu, \RR)$ and
 $$
 \Lambda \int f = \int \Lambda  f.
 $$
\end{theorem}

\begin{proof}
 If $\Lambda$ is a bounded linear functional on $\EE$, then
 $$
 |\Lambda v| \leq M (\|v\|_{\alpha_1}+\dots+\|v\|_{\alpha_k})
 $$
 for some $\alpha_1,\dots,\alpha_k\in I$ and some constant $M$ and all $v\in \EE$. Let
 $$
 p(v)=M (\|v\|_{\alpha_1}+\dots+\|v\|_{\alpha_k}).
 $$
 
 If $f=v_1A_1+\dots+v_nA_n$ for some $v_1, \dots , v_n \in \EE$ and disjoint $A_1, \dots , A_n \in \RI$, then
\begin{align*}
 \Lambda \int f &= \Lambda \left(\int (v_1A_1+\dots+v_nA_n)\right)\\
 &= \Lambda \left(v_1\mu(A_1)+\dots+v_n\mu(A_n)\right)\\
 &= \Lambda v_1 \mu(A_1)+\dots+\Lambda v_n \mu(A_n)\\
 &= \int \left(\Lambda v_1 A_1+\dots+\Lambda v_n A_n\right)=\int \Lambda f
\end{align*}
and
\begin{align*}
 \| \Lambda  f\|_1 &= \int |\Lambda (v_1A_1+\dots+v_nA_n)|\\
&= |\Lambda v_1| \mu(A_1)+\dots+|\Lambda v_n| \mu(A_n)\\
&\leq p( v_1) \mu(A_1)+\dots+ p(v_n) \mu(A_n)=  \int p f,
\end{align*}
where by $pf$ we mean the composition of $p$ and $f$.

Now let
 $$
f \simeq  f_1 + f_2 + \cdots.
$$
We will show that
\begin{equation}\label{lambdazp}
 \Lambda  f \simeq  \Lambda  f_1 +p f_1 - p  f_1 + \Lambda  f_2+ p f_2-p f_2+ \cdots.
\end{equation}
Indeed, we have
\begin{align*}
 \|\Lambda  f_1\|_1 +2\|p f_1\|_1  &+ \|\Lambda  f_2\|_1+ 2\|p f_2\|_1 + \cdots\\
 &\leq 3\|p f_1\|_1  + 3\|p f_2\|_1+ \cdots\\
 &= M\left( \sum_{n=1}^\infty \|f_n\|_{\alpha_1}+\dots+\sum_{n=1}^\infty\|f_n\|_{\alpha_k} \right) <\infty
\end{align*}
and, if 
$$
|\Lambda f_1(x)| + 2|pf_1(x)|+|\Lambda f_2(x)| +2|pf_2(x)|+ \cdots < \infty
$$
for some $x\in X$, then
$$
\sum_{n=1}^\infty \|f_n(x)\|_{\alpha_j}< \infty,
$$ 
for $j=1,\dots,k$ and thus
$$
\lim_{m\to\infty} \left\| f(x)- \sum_{n=1}^mf_n(x)\right\|_{\alpha_j} = 0
$$
for $j=1,\dots,k$. Hence
$$
\lim_{m\to\infty} p \left( f(x)- \sum_{n=1}^mf_n(x)\right) = 0
$$
and consequently
$$
\lim_{m\to\infty}  \left( \Lambda f(x)- \sum_{n=1}^m \Lambda f_n(x)\right) = 0
$$
or
$$
\Lambda f(x)=\sum_{n=1}^\infty \Lambda f_n(x)
$$
for that $x\in X$.

From \eqref{lambdazp} we get both $\Lambda  f \in L^1(X,\mu, \RR)$ and $\Lambda \int f = \int \Lambda  f$.

\end{proof}

\end{document}